\documentclass[11pt, reqno]{amsart}
\usepackage[T1]{fontenc}
\usepackage{amssymb,amsthm,amsmath}
\usepackage{enumitem,footnote,color,caption,tikz}
\usepackage[colorlinks=True]{hyperref}

\hypersetup{
    colorlinks=true,
    linkcolor=purple,
    filecolor=magenta,      
    urlcolor=cyan,
    citecolor=teal
    }

\usepackage{tikz}
\usetikzlibrary{quotes,angles,3d,calc,math}

\evensidemargin\oddsidemargin

\newtheorem{thm}{Theorem}%[section]
\newtheorem{lem}{Lemma}
\newtheorem{prop}[lem]{Proposition}
\newtheorem{cor}[lem]{Corollary}
\newtheorem{defn}[lem]{Definition}
\newtheorem{ex}[lem]{Example}

\DeclareMathOperator{\cov}{cov}

\newcommand{\R}{\mathbb{R}}
\newcommand{\N}{\mathbb{N}}

\newcommand{\al}{\alpha}
\newcommand{\be}{\beta}
\newcommand{\Z} {\mathbb{Z}}

\renewcommand{\leq}{\leqslant}
\renewcommand{\geq}{\geqslant}

\begin{document}

\baselineskip=15pt

\title[Discrete harmonic polynomials in the orthant]{Discrete harmonic polynomials in multidimensional orthants}

\date{\today}

\author{Emmanuel Humbert}

\address{Université de Tours, CNRS, Institut Denis Poisson, Tours, France. {\em Email}: {\tt emmanuel.humbert@univ-tours.fr}}

\author{Kilian Raschel} 

\address{Université d'Angers, CNRS, Laboratoire Angevin de Recherche en Mathématiques, Angers, France {\em Email}: {\tt raschel@math.cnrs.fr}}



\begin{abstract} 
We consider multidimensional random walks in pyramidal cones (or multidimensional orthants), which are intersections of a finite number of half-spaces. We explore the connection between the existence of (positive)\ discrete harmonic polynomials for the random walks, with Dirichlet conditions on the boundary of the cone, and geometric properties of the cone, being or not the Weyl chamber of a finite Coxeter group. We prove that the first property implies the second, derive the converse in dimension two and show in this case that it coincides with the probabilistic harmonic function. 
\end{abstract}

\maketitle

\section{Introduction and main results}
\label{sec:intro}

In this paper we consider random walks $(S(n))_{n\geq 0}$ in multidimensional cones $C\subset \mathbb R^d$, with $d\geq 2$. Our random walks $S(n)=X(1)+\cdots +X(n)$ are defined as follows: $X(1),\ldots,X(n)$ are independent, identically distributed random variables with values in a given finite set $\mathcal S\subset \mathbb R^d$. The distribution of $X(i)$ is as follows: for any $s\in\mathcal S$, $\mathbb P(X(i)=s)=a(s)$, where the weights or transition probabilities $a(s)$ are non-negative and such that 
\begin{align} 
\label{eq:normalisation}
&\sum_{s\in \mathcal S} a(s) =1 &(\text{normalization})\\
\label{sum_s}
&\sum_{s \in \mathcal S} a(s) s =0 &(\text{zero drift})\\
\label{sum_sisj}
&\sum_{s \in \mathcal S} a(s) s_i s_j =\delta_{i,j} &(\text{identity covariance matrix})
\end{align}
where in the latter equation $(s_1,\ldots,s_d)$ denote  the coordinates of $s \in \mathcal S$ in the canonical basis of $\R^d$, $\delta_{i,j}$ is the classical Kronecker symbol, and $i,j \in \{1,\ldots,d\}$.
%The condition \eqref{sum_s} can be also written as: for all $i$, $\sum_{s\in \mathcal{S}} a(s) s_i =0$.  
The random walk started at $x\in C$ is simply $(x+S(n))_{n\geq 0}$.
The discrete Laplacian $L_\mathcal S$ is defined for any function $f:\mathbb R^d\to\mathbb R$ by 
\begin{equation}
    \label{eq:discrete_Lap}
    L_\mathcal S(f)(x) := \sum_{s \in \mathcal S} a(s) f(x+s) -f(x).
\end{equation}
A function $f$ is said to be (discrete)\ harmonic on a domain $D$ if for all $x\in D$, $L_\mathcal S(f)(x)=0$.

When studying random walks in cones $C\subset \mathbb R^d$ (meaning random walks killed when exiting the cone $C$), it is natural to consider (positive)\ discrete harmonic functions with Dirichlet boundary conditions on the boundary of the cone. 
%Let us briefly review some links between harmonic functions and random walks in cones. 
Indeed, given such a harmonic function, one can use the classical Doob transformation to define a version of the random walk conditioned to stay in the cone. Defining such conditioned random walks is generally difficult, since a more direct approach based on the first exit time $\tau_C$ from the cone would imply conditioning on the zero probability event $\{\tau_C=\infty\}$. In another direction, in the context of Brownian motion, a crucial role is played by the réduite of the cone which, by definition, is the unique positive harmonic function in the kernel of the Laplacian operator $\Delta=\sum_{i=1}^{d}\frac{\partial^2}{\partial x_i^2}$ with Dirichlet boundary condition (uniqueness holds up to positive multiples).
As proved in \cite{DB-86,BaSm-97}, the réduite allows to estimate the (Brownian)\ survival probability in the cone and the heat kernel of the cone. Returning to random walks, in \cite{Va-99} Varopoulos derives several estimates of the survival probability of the random walk in the cone in terms of the réduite. The paper \cite{Va-99} also gives upper and lower bounds on the local probability $\mathbb P(x+S(n)=y,\tau_C>n)$ in terms of the réduite, where $x,y\in C$ and $\tau_C$ is the first exit time from the cone for the random walk started at $x$. In \cite{DeWa-15}, Denisov and Wachtel turn these upper and lower bounds into exact asymptotic statements using the discrete harmonic function. As in the continuous setting, one can indeed speak of \textit{the} discrete positive harmonic function, since it has been proved in \cite{DuRaTaWa-22}, using Martin boundary theory, that there is uniqueness of such a function (up to positive multiples). As shown in \cite{DeWa-15}, the leading term of the discrete harmonic function is harmonic in the sense of the classical Laplacian operator $\Delta$, but there is no reason in general that the discrete and continuous functions coincide.

In this paper we will consider a particular family of cones given by (convex)\ multidimensional orthants 
\begin{equation}
    \label{eq:def_orthant}
    C =Z(\R_+^d)
\end{equation} 
where $Z$ is an automorphism of $\R^d$. In dimension $d=2$ they are wedges, and in any dimension they are called pyramids. Our main motivation for looking at this subclass of cones is that they appear in many concrete situations. First, in enumerative combinatorics it is known that many models are naturally in bijection with walks in the quarter plane \cite{BMMi-10}, and more generally with lattice walks in multidimensional orthants $\mathbb Z_+^d$, see \cite{BoBMKaMe-16}. Note that the role of the transformation $Z$ in the definition \eqref{eq:def_orthant} of the cone is to transform the covariance matrix of the random walk into the identity matrix (as in \eqref{sum_s} and \eqref{sum_sisj}), in the domain of attraction of a standard Brownian motion; in this way the classical orthant $\R_+^d$ is transformed into the new cone $C =Z(\R_+^d)$. Second, there are several famous models that naturally reduce to random walks in orthants $C$ as in \eqref{eq:def_orthant}: non-colliding (or non-intersecting, or ordered) random walks and random walks in Weyl chambers, which appear in physics \cite{Fi-84}, combinatorics \cite{GeZe-92,GrMa-93,Fe-14,Fe-18,MiSi-20}, probability theory \cite{KoOCRo-02,EiKo-08,KoSc-10,DeWa-10,De-18}.

Now that we have defined the random walks and cones considered in this paper, we can introduce our main object of study, namely discrete polynomial harmonic functions. Although they haven't been studied much in the literature, polynomial harmonic functions contain a lot of structure. Indeed, in some examples they are related to representation theory and can be expressed in terms of the dimensions of certain representations \cite{Bi-91,Ra-11}; in other examples they are equal to the Vandermonde determinant \cite{EiKo-08}; they are related to conformal mappings in \cite{Ra-14}, etc. From a more combinatorial viewpoint, polynomial harmonic functions appear in the asymptotics of orbit-summable quadrant walks \cite{CoMeMiRa-17,Ne-24}.
We will study here discrete harmonic polynomials which vanish on $\partial C$, the boundary of $C$, and which are asymptotically positive on $C$, in the following sense:
\begin{defn}
\label{def:asymp_pos}
A polynomial $P$ is said to be {\em asymptotically positive} on $C$ if its dominant term is positive on $C$. For short, a polynomial which is discrete harmonic and asymptotically positive on $C$ will be called  an {\em ${\text{a}}^+$-polynomial on $C$}.
\end{defn}

  Our main result is the following (for classical definitions and facts on Coxeter groups, we refer to \cite{Hu-90}):

\begin{thm}
\label{harmonic0}
If there exists an ${\text{a}}^+$-polynomial $P$ on  $C$, then $P$ is unique and $C$ is a Weyl chamber of a finite Coxeter group. If $d=2$ and $C$ is a Weyl chamber of a finite Coxeter group (i.e., if the interior angle of $C$ has the form $\pi/n$ for some positive integer $n$), then there exists an ${\text{a}}^+$-polynomial on  $C$.  
\end{thm}

Let us make three important remarks about Theorem~\ref{harmonic0}.
First, suppose there exists a polynomial $P$ as in Theorem~\ref{harmonic0} which is not positive on $C$. Corollary~\ref{P_harm_H_finite} then implies that there is no positive discrete harmonic polynomial on $C$ which is vanishing on the boundary.

\medskip

Second, while Theorem~\ref{harmonic0} is clearly motivated by random walk theory and is stated using a Laplacian operator \eqref{eq:discrete_Lap} which is naturally associated with a random walk, it could alternatively and equivalently be stated as a result of discrete analysis, in the tradition of the early work of the theory, see e.g.\ Ferrand \cite{Fe-44}, Murdoch \cite{Mu-65} and Duffin \cite{Du-68}.

\medskip

Third, let us stress that our assumptions \eqref{eq:normalisation}, \eqref{sum_s} and \eqref{sum_sisj} are not restrictive. Indeed, let us consider arbitrary random walks restricted to a cone $C$, with associated jumps (or transition probabilities)\ that satisfy \eqref{eq:normalisation} and \eqref{sum_s} (if the drift turns out to be non-zero, one can easily perform a Cram\'er transformation to turn the walk into a zero-drift one). Now define $\cov$ as the covariance matrix of the associated random walk (see \eqref{covmatrix} for a precise definition). It is a positive definite and symmetric matrix.  So we can define $Z= \cov^{-\frac12}$, which defines an automorphism of $\R^d$. Then we introduce $\mathcal S'= Z(\mathcal S)$ and $C'=Z(C)$ (as in \eqref{eq:def_orthant} if the initial cone is the positive orthant $\mathbb R_+^d$). It is proved in \cite{DeWa-15} that the new step set $\mathcal S'$ satisfies the relations \eqref{eq:normalisation}, \eqref{sum_s} and \eqref{sum_sisj}. In other words, up to a linear transformation (which also affects the cone), we can always assume that \eqref{eq:normalisation}, \eqref{sum_s} and \eqref{sum_sisj} are satisfied.

\medskip

In connection with the above observation, we can now state a version of Theorem~\ref{harmonic0} in the special case of the cone $C=\mathbb R_+^d$. This statement is particularly useful in the combinatorial setting introduced in \cite{BMMi-10,BoBMKaMe-16}. In these cases the step set $\mathcal S$ is often supposed to have small steps, i.e., $\mathcal S \subset \{-1,0,1\}^d$. See Figure~\ref{fig:step_sets_2D} for four  illustrations of Theorem~\ref{harmonic}.
\begin{thm} 
\label{harmonic}
If there exists an ${\text{a}}^+$-polynomial for the discrete Laplacian operator \eqref{eq:discrete_Lap}  on  $\R_+^d$
and  $C$ is the multidimensional orthant as in \eqref{eq:def_orthant}, with $Z=\cov^{-\frac{1}{2}}$, then $C$ is a Weyl chamber of a finite Coxeter group. Conversely, if $d=2$ and $C=\cov^{-\frac{1}{2}}(\mathbb R_+^2)$ is a Weyl chamber of a finite Coxeter group (i.e., if the interior angle of $C$ has the form $\pi/n$ for some positive integer $n$), then there exists an ${\text{a}}^+$-polynomial on $\R_+^2$. 
\end{thm} 
 
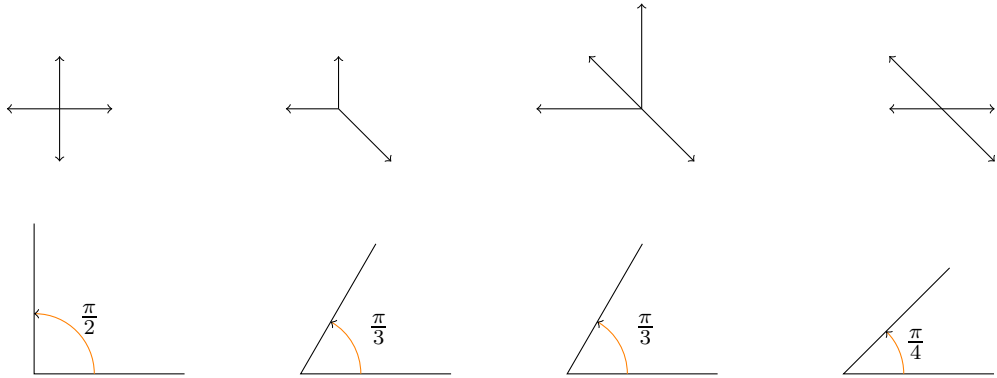
\begin{figure}[ht!]
\begin{center}
\begin{tikzpicture}[scale=.7] 
    \draw[->,white] (1,2) -- (0,-2);
    \draw[->,white] (1,-2) -- (0,2);
    \draw[->] (0,0) -- (1,0);
    \draw[->] (0,0) -- (-1,0);
    \draw[->] (0,0) -- (0,-1);
    \draw[->] (0,0) -- (0,1);
   \end{tikzpicture}\qquad\qquad\qquad
\begin{tikzpicture}[scale=.7] 
    \draw[->,white] (1,2) -- (0,-2);
    \draw[->,white] (1,-2) -- (0,2);
    \draw[->] (0,0) -- (-1,0);
    \draw[->] (0,0) -- (0,1);
    \draw[->] (0,0) -- (1,-1);
   \end{tikzpicture}\qquad\qquad\quad
\begin{tikzpicture}[scale=.7] 
  \draw[->,white] (1.5,2) -- (-2,-2);
    \draw[->,white] (1.5,-2) -- (-2,2);
        \draw[->] (0,0) -- (1,-1);
    \draw[->] (0,0) -- (-2,0);
    \draw[->] (0,0) -- (-1,1);
    \draw[->] (0,0) -- (0,2);
   \end{tikzpicture}\qquad\qquad
\begin{tikzpicture}[scale=.7] 
  \draw[->,white] (1.5,2) -- (-2,-2);
    \draw[->,white] (1.5,-2) -- (-2,2);
        \draw[->] (0,0) -- (1,0);
        \draw[->] (0,0) -- (-1,0);
    \draw[->] (0,0) -- (-1,1);
    \draw[->] (0,0) -- (1,-1);
   \end{tikzpicture}
  \end{center}
  \begin{center}
  \begin{tikzpicture}
  \draw
    (2,0) coordinate (a) %node[right] %{a}
    -- (0,0) coordinate (b) %node[left] %{b}
    -- (0,2) coordinate (c) %node[above right] %{c}
    pic["$\frac{\pi}{2}$", draw=orange, ->, angle eccentricity=1.3, angle radius=0.8cm]
    {angle=a--b--c};
\end{tikzpicture}\qquad\qquad
\begin{tikzpicture}
  \draw
    (2,0) coordinate (a) %node[right] %{a}
    -- (0,0) coordinate (b) %node[left] %{b}
    -- (1,1.73) coordinate (c) %node[above right] %{c}
    pic["$\frac{\pi}{3}$", draw=orange, ->, angle eccentricity=1.5, angle radius=0.8cm]
    {angle=a--b--c};
\end{tikzpicture}\qquad\qquad
\begin{tikzpicture}
  \draw
    (2,0) coordinate (a) %node[right] %{a}
    -- (0,0) coordinate (b) %node[left] %{b}
    -- (1,1.73) coordinate (c) %node[above right] %{c}
    pic["$\frac{\pi}{3}$", draw=orange, ->, angle eccentricity=1.5, angle radius=0.8cm]
    {angle=a--b--c};
\end{tikzpicture}
\qquad\qquad
\begin{tikzpicture}
  \draw
    (2,0) coordinate (a) %node[right] %{a}
    -- (0,0) coordinate (b) %node[left] %{b}
    -- (1.41,1.41) coordinate (c) %node[above right] %{c}
    pic["$\frac{\pi}{4}$", draw=orange, ->, angle eccentricity=1.3, angle radius=0.8cm]
    {angle=a--b--c};
\end{tikzpicture}
\end{center}
  \caption{Illustration of Theorem~\ref{harmonic}. Leftmost model: the function $P(i,j)=ij$ is harmonic for the simple random walk (with uniform probabilities $\frac14$). Second model (called tandem walk), with uniform probabilities $\frac13$: the polynomial $P(i,j)=ij(i+j)$ is discrete harmonic. Third model, with probabilities $a(1,-1)=\frac{1}{2}$ and the other probabilities are $\frac16$: the previous function is also harmonic. Rightmost model (with uniform probabilities $\frac14$): the function $P(i,j)=ij(i+j)(i+2j)$ is discrete harmonic. }
  \label{fig:step_sets_2D}
\end{figure}

We now want to   compare the harmonic functions appearing in our Theorems~\ref{harmonic0} and~\ref{harmonic} with a certain canonical discrete harmonic function useful in the theory of random walks in cones. As explained above, without restricting the generality, we can consider, as in the statement of Theorem~\ref{harmonic}, that $C= \R_+^d$ and that $\mathcal S$ satisfies the normalisation and zero drift assumptions \eqref{eq:normalisation} and \eqref{sum_s} but not necessarily \eqref{sum_sisj}.   Then Denisov and Wachtel proved in \cite[Eq.~(6)]{DeWa-15} that there exists a function $f$  satisfying the following harmonic equation, for all $x\in \R_+^d$ 
\begin{equation}
\label{eq:harm_rel_DW}
 f(x) = \mathbb E_{x}\bigl[f\bigl(x+X(1)\bigr),\tau_C >1\bigr],
\end{equation}
which is positive on $(0,\infty)^d$ and vanishing on the boundary. Here  $\tau_C$ denotes the first exit time from $C=\R_+^d$.  It turns out that $f$ is indeed unique, as proved in \cite[Thm~4]{DuRaTaWa-22}. 

\medskip

It is then natural to ask whether or not $f$ in \eqref{eq:harm_rel_DW} is equal to the harmonic function obtained in Theorem~\ref{harmonic} (if it exists). In a general setting this cannot be true: indeed, $f$ is defined on $\R_+^d$ and the equation \eqref{eq:harm_rel_DW} is equivalent to $L_\mathcal S(f)=0$ only if the function $f$ is extended by $0$ outside $\R_+^d$. So it is not polynomial. However, from a probabilistic point of view, when studying lattice random walks, only the values of $f$ in $\N_0^d$ are relevant (here and throughout we will denote $\N_0=\{0,1,2,\ldots\}$ and $\N=\{1,2,3,\ldots\}$). 

A natural question is whether it is true that $f=P$ on $\N^d$. Since both $P$ and $f$ vanish at the boundary, this question makes sense as soon as the random walk cannot jump over the boundary. More precisely, the last condition means that for $x \in \N^d$ it holds that $x +s \in \N_0^d$ for any $s \in \mathcal S$ or equivalently that all Cartesian coordinates of $s$ belong to $\{-1,0,1,2,\ldots\}$.  All examples in Figure~\ref{fig:step_sets_2D} (except the third) satisfy this property. If this condition does not hold, $f$ is obviously not polynomial.

Let us illustrate this fact on the   following example, taken from \cite{BMFuRa-20}. Considering the third model in the Figure~\ref{fig:step_sets_2D}, which does not satisfy the small negative jump hypothesis, it turns out that the function $P_0(i,j)=ij(i+j)$ is discrete harmonic on $\mathbb Z^2$. However, it is easy to verify that while it satisfies the harmonicity property \eqref{eq:harm_rel_DW} for almost all points, it does not satisfy it for $i=1$, so it is not the probabilistic harmonic function for this model. In fact
 by \cite[Eq.~(63)]{BMFuRa-20}, the function below turns out to be the only positive harmonic function that satisfies \eqref{eq:harm_rel_DW}:
\begin{equation*}
    f(i,j)=j\left(i(i+j)+\frac{i}{2}+\frac{j}{4}-\frac{1}{8}-\frac{2j-1}{8}\left(-\frac{1}{3}\right)^i\right).
\end{equation*}

Now let us return to the case where the random walk cannot jump the boundary.
%We say that $f$ is {\em polynomial} there exists a polynomial $P$
%\begin{equation*}
 %   f(x) = \left\{\begin{array}{ccc}
 %   P(x) & \text{if} & x\in \N^d,\\
 %   0 & \text{if} & x \notin \N^d,
 %   \end{array}\right.
%\end{equation*}
Then we prove:
\begin{thm} \label{comparingfP}
Assume that  $d=2$ and $C:=\cov^{-\frac{1}{2}} \left(\R_+^d\right)$ is the Weyl chamber of a finite Coxeter group. Then, the harmonic function $f$ of \cite{DeWa-15} in \eqref{eq:harm_rel_DW} coincides on $\N^2$ with the harmonic polynomial given by Theorem \ref{harmonic}, up to a positive multiplicative constant. 
\end{thm}

%As a corollary, we obtain when $d=2$ and  $C=\cov^{-\frac{1}{2}} \left(\R_+^d\right)$ is the Weyl chamber of a finite Coxeter group, the harmonic polynomial $P$ given by Theorem \ref{harmonic} is positive on $\N^2.$ But actually, we can get :
%\begin{cor} \label{cor_pos}
%Assume that $d=2$ and that  $C=\cov^{-\frac{1}{2}} \left(\R_+^d\right)$ is the Weyl chamber of a finite Coxeter group. Then, the harmonic polynomial $P$ given by Theorem \ref{harmonic} is positive on $]0,+\infty[^2.$
%\end{cor}

We now recall the construction of a certain canonical (continuous) harmonic function $P_0$, mentioned as the réduite of the cone at the beginning of the introduction, which in certain cases can also be discrete harmonic, as we will explain in Theorem~\ref{P0harm} below. By definition, the walls of $C$ are  the $d$ hyperplanes 
 $$F_i = {\rm{span}}\bigl\{Z(e_1), \ldots,\widehat{Z(e_i)}, \ldots,Z(e_d)\bigr\},$$ 
where classically $(e_1,\ldots,e_d)$ denotes the canonical basis of $\mathbb R^d$, and the vector with a hat is absent from the set.
 Assume that $C$ is the Weyl chamber of a finite Coxeter group $H$. Let $\{r_1,\ldots,r_d\}$ be the reflections through the walls $F_i$ of $C$. Then $\{r_1,\ldots,r_d\}$ is a set of generators of $H$. 
 
For any hyperplane $F$ of $\R^d$, denote by $r_F$ the reflection through $F$.  Define further
 $$\Lambda:=\bigl\{F | F \hbox{ is a hyperplane of } \R^d \hbox{ such that } r_F \in H \bigr\}.$$
 Then $F_1, \ldots,F_d \in \Lambda$.
When $H$ is finite and $C$ is a Weyl chamber of $H$, then we let $F_{d+1},\ldots,F_r$ be such that 
$\Lambda=\{F_1,\ldots,F_r\}$. For each $i$, we choose a linear form $f_i$ such that 
\begin{equation}
    \label{eq:def_Fifi}
    F_i:=\ker{f_i}
\end{equation} and we set 
\begin{equation} 
\label{defP0}
   P_0= \prod_{i=1}^r f_i.
\end{equation}
Note that the polynomial $P_0$ depends on the choice of the $f_i$ up to a multiplicative constant. Without loss of generality, we can assume that $P_0>0$ on $C$. Note also that $P_0$ is homogeneous, in the sense that $P_0(\lambda x)=\lambda^r P_0(x)$ for any $x\in \mathbb R^d$  and $\lambda>0$. As explained in \cite{GoHuRa-25}, $P_0$ is harmonic for the classical Laplacian $\Delta$ and is antisymmetric with respect to $H$, i.e., satisfies for all $h \in H$ that
\begin{equation} 
\label{P0antisym}
   P_0 \circ h = \det(h) P_0.
\end{equation} 
The polynomial $P_0$ is a natural candidate for the discrete harmonic polynomial given by Theorem~\ref{harmonic0}. More generally, a natural question is whether or not $P_0$ is a discrete harmonic polynomial.  Obviously this is not true in general (see for example \cite{Ra-11}), but 
we show that:
\begin{thm} 
\label{P0harm} 
Assume that the set $\mathcal S$ and the associated weights $a(s)$  are invariant under the action of $H$ (i.e., $h\mathcal S=\mathcal S$ and $a(hs)=a(s)$ for all $h \in H$ and $s \in \mathcal S$), then $P_0$ in \eqref{defP0} is discrete harmonic. 
\end{thm}

As a first illustration of Theorem~\ref{P0harm}, the function $P_0$ is discrete harmonic for all models represented on Figure~\ref{fig:step_sets_2D}. As a second illustration, consider the rightmost step set on Figure~\ref{fig:step_sets_2D}, with the following transition probabilities
\begin{equation*}
    a(1,0)=a(-1,0) = \frac{\sin^2(\frac{\pi}{n})}{2}\quad \text{and}\quad a(1,-1)=a(-1,1) = \frac{\cos^2(\frac{\pi}{n})}{2},
\end{equation*}
where $n\geq 3$ is a fixed parameter. Theorem~\ref{harmonic} entails that there exist  harmonic polynomials $P_n$ if and only if $n$ is an integer. Moreover, a family of discrete harmonic polynomials  $P_n$ is constructed in \cite{Ra-11}: 
\begin{equation*}
    \left\{\begin{array}{rcl}
    P_3(i,j) & = & ij(i+2j),\\
    P_4(i,j) & = & ij(i+2j)(i+j),\\
    P_6(i,j) & = & ij(i+2j)(i+j)\bigl((i+\tfrac{2}{3}j)(i+\tfrac{4}{3}j)+\tfrac{10}{9}\big).
    \end{array}\right.
\end{equation*}
Observe that $P_3$ and $P_4$ are homogeneous polynomials, contrary to $P_6$. It would be interesting to characterise all Laplacian operators and weights that lead to homogeneous polynomial harmonic functions.

A natural question then is: are there models $\mathcal{S}$ such that $C$ is a Weyl chamber of a finite Coxeter group? In \cite{GoHuRa-25} we prove that this is equivalent to the fact that the non-diagonal coefficients of the covariance matrix $\cov$ belong to a short list of numbers. Thus, constructing such models is equivalent to prescribing the matrix $\cov$. In Section~\ref{example} we show that any positive definite and symmetric matrix (in dimension $d=2$ and $3$) can be realised as the covariance matrix of a random walk, and the method gives a constructive way to obtain families of models for which $C$ is a Weyl chamber of a finite Coxeter group. See our Theorem~\ref{cov_presc}.

\section{Proof of Theorem~\ref{harmonic0}}
The proof consists of several steps. First, recall the following classical lemma; remember that $\Delta$ denotes the classical Laplacian operator on $\mathbb R^d$, and that the transition probabilities $a(s)$ introduced in Section~\ref{sec:intro} satisfy \eqref{eq:normalisation}, \eqref{sum_s} and \eqref{sum_sisj}.

\begin{lem}
\label{laplacian} 
For any smooth function $f : \R^d \to \R$ and all $x \in\R^d$,
$$\lim_{h \to 0} \frac{1}{h^2}\left(\sum_{s \in \mathcal S} a(s) f(x +hs) - f(x) \right) = \frac{1}{2} \Delta f (x) .$$
\end{lem}

\begin{proof} 
%for all $s\in \mathcal{S}$
%$$ \begin{aligned} 
%f(x+hs) & = f(x) + h {\rm{d}} %f_x(s) +\frac{h^2}{2} d^2f_x(s,s) + o(h^2) \\
%& = f(x) + h {\rm{d}} f_x(s) + \frac{h^2}{2} \sum_{1 \leq i,j \leq d} \partial_{ij}f(x)  s_is_j + o(h^2) \\
%& = f(x) +  h {\rm{d}} f_x(s) + h^2 \sum_{1 \leq i < j  \leq q } \partial_{ij}f(x)  s_is_j + \frac{h^2}{2} \left( \sum_{i=1}^d
%s_i^2\right) \partial_{ii} f(x)+o(h^2). \end{aligned}$$
Using a Taylor expansion of 
$f(x+hs)$ at $h=0$,
multiplying by $a(s) $ and summing for $s \in \mathcal S$, we get 
\begin{multline*}
\sum_{s\in \mathcal{S}} a(s) f(x +hs) -f(x) = h\,{\rm{d}}  f_x\left( \sum_{s\in \mathcal{S}} a(s)s \right)  \\
 +  h^2 \sum_{1 \leq i < j  \leq d} \partial_{ij}f(x) \sum_{s\in \mathcal S} a(s)  s_is_j +  \frac{h^2}{2}  \sum_{i=1}^d
\left( \sum_{s\in \mathcal{S}} a(s) s_i^2 \right) \partial_{ii} f(x) + o(h^2). \end{multline*}
Using the assumptions \eqref{sum_s} and  \eqref{sum_sisj}, we finally obtain that 
\begin{equation*}
    \sum_{s \in \mathcal S} a(s) f(x +hs) -f(x)  = \frac{h^2}{2} \Delta f(x) + o(h^2).\qedhere
\end{equation*}
\end{proof}

Lemma~\ref{laplacian} compares the discrete and continuous Laplacian operators when the mesh of the lattice goes to zero. In the following result we compute the exact difference between these two operators. 

\begin{lem}\label{Ls=}
For any polynomial $P:\mathbb R^d\to\mathbb R$, it holds that 
$$L_{\mathcal{S}}(P) = \frac{1}{2} \Delta (P) +\sum_{k \geq 3}  \sum_{s\in \mathcal{S}} \frac{a(s)}{k!}   {\rm{d}}^k_xP(s,\ldots,s).$$  \end{lem}
\begin{proof}
The Taylor expansion of the polynomial $P$ gives
$$L_{\mathcal{S}}(P) = \sum_{k \geq 1}  \sum_{s\in \mathcal{S}} \frac{a(s)}{k!}   {\rm{d}}^k_xP(s,\ldots,s).$$  
As in the proof of Lemma~\ref{laplacian}, we get from \eqref{sum_s} that 
$\sum_{s\in \mathcal{S}} a(s)  {\rm{d}}_xP(s)=0$ and from \eqref{sum_sisj} that 
$$\sum_{s\in \mathcal{S}} a(s)  {\rm{d}}^2_xP(s,s)= 2  \sum_{s\in \mathcal{S}} a(s)  \sum_{1 \leq i < j  \leq d } \partial_{ij}P(x)  s_is_j +\sum_{s\in \mathcal{S}} a(s)  \left( \sum_{i=1}^d
s_i^2\right) \partial_{ii} P(x) =  \Delta P(x),$$
which proves Lemma~\ref{Ls=}.
\end{proof}

We recall that, when $C$ is the Weyl chamber of a finite reflection group $H$, 
$P_0=\prod_{i=1}^r f_i$ is defined by \eqref{defP0}, is harmonic and positive in $C$ (since $C$ is a Weyl chamber of the reflection group $H$).
We also set 
\begin{equation}
\label{eq:def_P1}
    P_1=\prod_{i=1}^d f_i,
\end{equation} 
where we recall that $F_1,\ldots,F_d$ are the walls of $C$, see Section~\ref{sec:intro}. Note that $P_1$ is well defined even if $H$ is infinite.

The following lemma is a special case of \cite[Lem.~13]{DeWa-15}; we prove it in our special case in order to make the paper self-contained, and also because we will use the same type of computations later.  
\begin{lem}
\label{dominant_term}
Assume that $P$ is a discrete harmonic polynomial on $C$. Then its dominant term is harmonic in $C$ and vanishes on each $F_i$, $i=1, \ldots, d$. 
\end{lem}

\begin{proof}
We write $P$ as $P=D + \sum_{i=1}^{k-1} R_i$, where $D$ is homogeneous of degree $k=\deg P$ and  $R_i$ is homogeneous of degree $i <k$. Note that $k \geq d$ since, as $P$ vanishes on all $F_i$ for each $i \in \{1,\ldots,d\}$, it holds that $P_1$ in \eqref{eq:def_P1} divides $P$.
Moreover, since $P$ is discrete harmonic, we have that for all $h \not= 0$,
$$\begin{aligned} 
0 & = h^{k-2} \left( \sum_{s\in \mathcal{S}} a(s) P\left(\frac{x}{h} +s\right) -P\left(\frac{x}{h}\right) \right) \\
& = h^{k-2} \left(\sum_{s\in \mathcal{S}} a(s) D\left(\frac{x}{h} +s\right) -D\left(\frac{x}{h}\right)\right) + h^{k-2}  \sum_{i=1}^{k-1} \left(\sum_{s\in \mathcal{S}} a(s) R_i\left(\frac{x}{h} +s\right) -R_i\left(\frac{x}{h}\right) \right)  \\
& = \frac{1}{h^2} \left(\sum_{s\in \mathcal{S}} a(s)D(x +sh) -D(x)\right) + \sum_{i=1}^{k-1} h^{k-i-2}   \left(\sum_{s\in \mathcal{S}} a(s) R_i(x +sh) -R_i(x)\right).\end{aligned} $$
From Lemma~\ref{laplacian}, for all $i <k$, when $h$ tends to $0$
$$h^{k-i-2}   \left(\sum_{s\in \mathcal{S}}  a(s) R_i(x +sh) -R_i(x)\right)= O(h^{k-i})$$
and 
$$\lim_{h \to 0} \frac{1}{h^2} \left(\sum_{s\in \mathcal{S}} a(s)D(x +sh) -D(x)\right)= \frac{1}{2} \Delta D(x),$$ which shows that $D$ is harmonic. 

Let finally $x \in F_i$ for some $i \in \{1,\ldots,d\}$. Then for all $h>0$ we have
$\frac{x}{h} \in F_i$ and thus 
\begin{equation*}
0 = h^{k}P\left(\frac{x}{h}\right) \stackrel{h \to 0}{\longrightarrow} D(x).\qedhere\end{equation*}

\end{proof} 
An immediate corollary of the above lemmas together with \cite[Thm~14]{GoHuRa-25} is:
\begin{cor} \label{P_harm_H_finite}
 Assume that there exists an ${\text{a}}^+$-polynomial $P$ in $C$. Then $H$ is finite and $C$ is a Weyl chamber of $H$. In this case, $P$ is unique up to a multiplicative constant and has the form $P= P_0 + R$, where $\deg R < \deg P.$  
\end{cor}

\begin{proof} 
Write $P=D+R$, where $D$ is the dominant term of $P$. From Lemma~\ref{dominant_term},  $D$ is a harmonic homogeneous polynomial in $C$, vanishing on each $F_i$ and positive on the interior of $C$ (since $P$ is asymptotically positive, see Definition~\ref{def:asymp_pos}). Its restriction to $\mathbb{S}^{d-1}$ is thus an eigenfunction $\varphi$ of the sphere and $C \cap \mathbb{S}^{d-1}$ is a nodal domain of $\varphi$. Using a theorem proved in our earlier work \cite[Thm~14]{GoHuRa-25}, we conclude that $H$ is finite and $C$ is a Weyl chamber of $H$. Also, from the proof of \cite[Thm~14]{GoHuRa-25}, we deduce that $D$ has the form $D= \alpha P_0$ for some $\al >0$. 

Now suppose $P'$ is any ${\text{a}}^+$-polynomial in $C$. Up to a multiplicative constant we can assume that they both have the same dominant term $P_0$. Then $P-P'$ is a discrete harmonic polynomial which vanishes on each $F_i$, $i=1,\ldots,d$. From Lemma~\ref{dominant_term}, the dominant term $D_0$ of $P-P'$ is harmonic and vanishing on all $F_i$, $i=1,\ldots,d$. We will prove that $P_0$ divides $D_0$, which will
 imply that $D_0=0$ (and thus $P'=P$) since $\deg(P-P') < \deg P_0$. 
 
 To prove that $D_0$ is divisible by $P_0$, we start by proving that $D_0$ is  antisymmetric with respect to $H$, i.e., that for all $h \in H$,
$D_0 \circ h = \det(h) D_0$. We proceed as in \cite[Thm~14]{GoHuRa-25}. Let $i \in \{1,\ldots,d\}$.
We recall that $r_i$ is the reflection with respect to $F_i$. Denote by $F_i^+$ the half-space whose boundary is $F_i$ and which contains $C$, and let $F_i^-= \R^d \setminus {\overline{F_i^+}}$. Define 
\begin{equation*}
    g : \left| \begin{array}{ccc} 
D_0 + D_0\circ r_i & \hbox{ on } & \overline{F_i^+}, \\
0 & \hbox{ on } & F_i^-. \end{array} \right. 
\end{equation*}
As in the proof of \cite[Thm~14]{GoHuRa-25}, we easily check that $g$ is of class $\mathcal C^2$ and harmonic. Since it vanishes on $F_i^-$, it must vanish everywhere. 
We deduce that $D_0\circ r_i = -D_0=\det(r_i) D_0$. Since $H$ is spanned by the $r_i$, $i=1, \ldots,d$, we get that $D_0$ is antisymmetric with respect to $H$. Since all $r_i$, $i=1,\ldots,r$, belong to $H$ and since $H$ is spanned by the $r_i$, $i=1, \ldots, d$, we deduce that for all $i =1,\ldots,r$, $D_0\circ r_i = -D_0=\det(r_i) D_0$ and thus $D_0$ vanishes on all $F_i$, $i=1,\ldots,r$. Arguing as in  \cite[Lem.~18]{GoHuRa-25}, we deduce that for all $i =1,\ldots,r$, 
$D_0$ is divisible by the linear form $f_i$, and since the ring of polynomials of degree less than $r$ is factorial, we deduce that 
$P_0=\Pi_{i=1}^r f_i$ divides $D_0$,
which
concludes the proof.
\end{proof} 

We thus proved one direction of Theorem~\ref{harmonic0}. We now assume that $C$ is a  Weyl chamber of the group $H$ defined above (which is then supposed to be finite). For $n \in \N_0$, we denote by 
$\R_n[X_1,\ldots,X_d]$ the space of polynomials of degree less or equal to $n$. We recall that its dimension is $\binom{n+d}{n+1}$ and  that $r$ is the degree of $P_0$, see \eqref{defP0}.
Define 
\begin{equation*}
    V= \bigl\{ P \in \R_{r-1}[X_1,\ldots,X_d]| P_1 \hbox{ divides } P \bigr\}.
\end{equation*}
Clearly the application
\begin{equation}
    \label{eq:application_iso}
    \begin{array}{ccc}
\R_{r-1-d}[X_1,\ldots,X_d] & \to & V \\
Q & \mapsto & P_1Q
\end{array}
\end{equation}
is an isomorphism and these spaces have the same dimension.

\begin{lem} 
\label{L_inj} 
The discrete Laplacian
$$L_{\mathcal{S}} : 
V \to  \R_{r-3}[X_1,\ldots,X_d] $$ 
introduced in \eqref{eq:discrete_Lap} is injective. 
\end{lem} 

\begin{proof}
From Lemma~\ref{Ls=}, $L_{\mathcal{S}}(V) \subset  \R_{r-3}[X_1,\ldots,X_d].$ It remains to prove that the map \eqref{eq:application_iso} is injective. On the contrary, let us assume the existence of $P \in V$ such that $L_{\mathcal{S}} P=0$. From Lemma~\ref{dominant_term}, its dominant term is harmonic and vanishing on every $F_i$. As in the proof of Corollary~\ref{P_harm_H_finite}, we get that it is a multiple of $P_0$, which is impossible since its degree is strictly less than $r$. A contradiction.
\end{proof}

\begin{lem} \label{LP0}
It holds that $L_{\mathcal{S}}(P_0) \in  \R_{r-3}[X_1,\ldots,X_d]$.
\end{lem}
\begin{proof} 
Direct consequence of the fact that $P_0$ is harmonic and of Lemma~\ref{Ls=}. 
\end{proof}

We are now ready to finish the proof of Theorem~\ref{harmonic0} (or equivalently of Theorem~\ref{harmonic}). 

\begin{proof}[Proof of Theorem~\ref{harmonic0}]
Assume that $d=2$. Then $V$ is isomorphic to 
\begin{equation}
\label{eq:eg_dim}
    \R_{r-d-1}[X_1,\ldots,X_d]= \R_{r-3}[X_1,\ldots,X_d].
\end{equation}
By Lemma~\ref{L_inj}, $L_{\mathcal{S}} : 
V \to  \R_{r-3}[X_1,\ldots,X_d] $ must be an isomorphism. In particular, from Lemma~\ref{LP0}, there exists $R \in V$ such that $L_{\mathcal{S}}R= -L_{\mathcal{S}}P_0$. We obtain that $L_{\mathcal{S}}(P_0 + R)=0$. The polynomial  $P=P_0+R$  is thus discrete harmonic and vanishes on each $F_i$ since $P_1 | (P_0 +R)$. This proves Theorem~\ref{harmonic0}. 
\end{proof}
Obviously the equality \eqref{eq:eg_dim} only holds in the two-dimensional case, which explains why we have to assume that $d=2$ in the converse statements of Theorems~\ref{harmonic0} and~\ref{harmonic}. It is an interesting open problem to ask whether one can extend these assertions to any dimension.

\section{Proof of Theorem~\ref{P0harm}}
We recall that $H$ is the reflection group spanned by the reflections with respect to the walls of $C$, see Section~\ref{sec:intro}. We assume that $C$ is a Weyl chamber of $H$ (and thus $H$ is finite). We further assume that for all $h \in H$ and $s \in \mathcal S$, it holds that $h (\mathcal S)=\mathcal S$ and $a(hs)=a(s)$. We prove that $P_0$ defined by \eqref{defP0} is then discrete harmonic.

Define for all $x$
$$P(x)= L_{\mathcal{S}}(P_0)(x) + P_0(x) = \sum_{s\in \mathcal{S}} a(s)P_0(x+ s).$$
From Lemma~\ref{Ls=} it follows that $P(x)$ is a polynomial whose dominant coefficient is $P_0$. So to prove that $L_{\mathcal{S}} P_0=0$, or equivalently that $P=P_0$, we just need to prove that $P_0$ divides $P$. To do this, we just have to check that $P$ vanishes on every $F_i \in \Lambda$. Recall from \eqref{defP0} that $P_0= \prod_{i=1}^r f_i$ and from \eqref{eq:def_Fifi} that $F_i=\ker(f_i)$. So we fix $i$ and show that $P$ vanishes on $F_i$. Recall that $r_i$ is the reflection through the hyperplane $F_i$. Let $H_0$ be the kernel of $\det : H \to \{-1, 1\}$. It is clear that $H = H_0 \cup H_0r_i$. Let $x \in F_i$. We write that 
$$P(x) = \sum_{s\in \mathcal{S}} a(s) P_0(x+s).$$
Since $\mathcal S$ and $a$ are invariant under the action of $H$, we get, using \eqref{P0antisym}
$$\begin{aligned} P(x)&  = \frac{1}{|H|} \sum_{h \in H} \sum_{s\in \mathcal{S}} a(h^{-1} s)  P_0(x + h^{-1} s)\\
 & = \frac{1}{|H|} \sum_{h \in H} \sum_{s\in \mathcal{S}} a(s) P_0(h^{-1}(h x + s)) \\
 & = \frac{1}{|H|} \sum_{h \in H} \sum_{s\in \mathcal{S}} a(s) \det(h) P_0(h x + s) \\
 & = \frac{1}{|H|} \sum_{h \in H_0} \sum_{s\in \mathcal{S}}a(s) \bigl( \det(h) P_0(h x + s) + \det(h \circ r_i) P((h \circ r_i) x +s) \bigr). \end{aligned}$$
 Now, since $x \in F_i$, $r_i(x)=x$, and for all $h \in H_0$, $\det(h)= -\det(h \circ r_i)=1$ and thus, 
 $$ \det(h) P_0(h x + s) + \det(h \circ r_i) P(h \circ r_i x +s)=0,$$
 which proves that $P(x)=0$ and completes the proof of Theorem~\ref{P0harm}.

\section{Proof of Theorem \ref{comparingfP}}
Let $P$ be the ${\text{a}}^+$-polynomial given by Theorem~\ref{harmonic}. It can be written as 
$P= P_0 + R$, where $|P| \leq c+ |x|^r$, $r$ is the degree of $P_0$ which is defined in the introduction of the paper. 
We use the notations of Section~\ref{sec:intro}: $(X(i))_{i\geq 1}$ is a sequence of independent, identically distributed random variables such that $\mathbb P(X(i)=s) =a(s)$ for all $s \in \mathcal S$. 
We let also $\tau_x$ be the first exit time from $\R_+^2$ of the random walk $(x+S(n))_{n\geq0}$ started at $x$.
We let $\varepsilon >0$, whose value will be fixed later, and define
$$\Omega_n^\varepsilon = \left\{ (i,j) \in \N^2 \Big| |(i,j)|= \max\{|i|,|j|\} \leq n^{\frac{1}{2}+ \varepsilon} \right\}.$$
In particular, 
$\sharp \Omega_n^\varepsilon \leq  n^{1+2 \varepsilon}$.
Let us fix $x \in \N^2$. For all $n \in \N$, we write that
\begin{equation} \label{E(R)}
\mathbb E \bigl[R(x+S(n)), \tau_x > n\bigl] = E_1 +E_2,
 \end{equation} 
 where 
 $$E_1= \int_{y \in \Omega_n^\varepsilon} R(y) {\rm d} \mathbb P (x+S(n)=y,\tau_x >n)\quad 
 \text{and} 
 \quad E_2 = \int_{y \in \N_0^2 \setminus \Omega_n^\varepsilon } R(y) {\rm d} \mathbb P (x+S(n)=y,\tau_x >n).$$
 By \cite[Thm~5]{DeWa-15}, there exists a constant $C>0$ such that  for all $y \in \N^2$,  
 \begin{equation} \label{P=}  
 \mathbb P (x+S(n)=y,\tau_x >n)  \leq C n^{- \frac{r}{2} - 1} f(x) P_0\left(\frac{y}{n^{\frac{1}{2}}} \right) e^{-\frac{|y|^2}{2n}} 
 \end{equation} 
and thus when $ y \in \Omega_n^\varepsilon$, we get (with a possibly different constant $C>0$)
$$\mathbb P (x+S(n)=y,\tau_x >n )\leq C |f(x)|  n^{- \frac{r}{2} -1 + r \varepsilon}.$$
In the following, the constant $C$ may change from line to line. Plugging the above estimate and the upper bound on the cardinality of $\Omega_n^\varepsilon$ into the definition of $E_1$, and since for all $y$ $|R(y)| \leq C( 1+ |y|^{r-1})$, we get that if $\varepsilon$ is small enough
$$ \begin{aligned}
E_1 & \leq C |f(x)| (\sharp \Omega_n^\varepsilon)  \max_{\Omega_n^\varepsilon} |R(y)|  n^{-\frac{r}{2} -1 + r \varepsilon } \\
& \leq C |f(x)| n^{1 + 2 \varepsilon} n^{\frac{r-1}{2}+(r-1) \varepsilon}  n^{-\frac{r}{2} -1 + r \varepsilon } \\
& \leq C |f(x)| n^{-\frac{1}{2}+\varepsilon(1 + 2r )}\\
& =o(1).
\end{aligned}$$
By \eqref{P=}, it also holds that when $y   \in \N_0^2 \setminus \Omega_n^\varepsilon $, there exists $C>0$ independent of $y$ such that 
$$\mathbb P (x+S(n)=y,\tau_x >n) \leq C e^{- n \varepsilon}$$ 
which implies that $\lim_{n \to  \infty} E_2=0$.
Together with \eqref{E(R)}, we obtain that for any fixed $x \in \N^2$
\begin{equation} \label{Eto0}
\lim_{n \to  \infty} \mathbb E \bigl[R(x+S(n)), \tau_x > n\bigl] =0.
\end{equation}

Since $P$ is harmonic, it holds that 
for all $n\in\N$ and $x \in \N^2$  
$$
P(x)  = \mathbb E\bigl[P(x+S(n)), \tau_x >n\bigl]  = \mathbb E\bigl[P_0(x+S(n)), \tau_x >n\bigl] + \mathbb E \bigl[R(x+S(n)), \tau_x >n\bigl].$$ (Note that the above equation is not true for all $x \in \R_+^2$, but needs $x$ to be a lattice point.)
By \cite[Lem.~12]{DeWa-15}, it holds that 
$$\lim_{n \to  \infty} \mathbb E\bigl[P_0(x+S(n)), \tau_x >n\bigl] =f(x).$$
Together with \eqref{Eto0}, we obtain that $P(x)=f(x)$, which proves Theorem~\ref{comparingfP}. We thank Pierre Tarrago for suggesting the idea of the proof.

\section{Prescription of the covariance matrix} 
\label{example} 
The aim of this section is to give a method to construct, in dimension $d=2$ and $3$, families of random walk models with a prescribed covariance matrix $\cov$, which lead to an orthant $C= \cov^{-\frac{1}{2}} \R_+^d$ in \eqref{eq:def_orthant} which is a Weyl chamber of a finite Coxeter group. In other words, with this construction at hand, our Theorems~\ref{harmonic0} and~\ref{harmonic} apply to random walks that can be defined explicitly.

As presented in the introduction, we consider here random walks defined by a finite set $\mathcal S \subset \{-1,0,1\}^d$ and some associated probabilities $a(s) $ such that $\sum_{s \in \mathcal S} a(s)=1$ and $\sum_{s \in \mathcal S}a(s)s=0$, see \eqref{eq:normalisation} and \eqref{sum_s}. For each such model, we define the inventory polynomial
$$\chi_{\mathcal S}(x_1,\ldots,x_d) = \sum_{s \in \mathcal S} a(s) x_1^{s_1}\cdots x_d^{s_d},$$
where $s=(s_1,\ldots,s_d)$. We set  $\boldsymbol{x_0}=(1,\ldots, 1)$, which is the unique critical point of $\chi_{\mathcal S}$ on $\R_+^d$. By definition and after a simple computation, the covariance matrix of the random walk is given by 
\begin{equation} 
\label{covmatrix} 
{\cov}(\chi_{\mathcal S}) = \left( \frac{\partial_{ij} \chi_{\mathcal S}(\boldsymbol{x_0}) }{\sqrt{\partial_{ii } \chi_{\mathcal S}(\boldsymbol{x_0}) \partial_{jj} \chi_{\mathcal S}(\boldsymbol{x_0})}} \right)_{1 \leq i,j \leq d}.
\end{equation}
%We denote by $(x_1,\ldots,x_d)$ the standard coordinates on $\R^d$.

 We start by studying the case of dimension $3$. 
Let $a,b,c  \geq 0$ and set 
\begin{equation}
    \label{eq:cov_dim3}
    \delta= \left( \begin{array}{ccc} 
1 & -a & -b \\
-a & 1 & -c \\
-b & -c & 1 \end{array} \right).
\end{equation}
The question we address in this section is to find a necessary and sufficient condition on $a,b,c$ for $\delta$ in \eqref{eq:cov_dim3} to be the covariance matrix of a random walk with small steps in $\mathbb Z^3$. Since ${\cov}(\chi_{\mathcal S})$ in \eqref{covmatrix}  is positive definite, its determinant
$$1 - (a^2+b^2+c^2 + 2 abc)$$ 
must be positive. In fact, we show that the converse statement is true:
\begin{thm} 
\label{cov_presc}
Let $a,b,c \geq 0$ and $\delta$ be defined as in \eqref{eq:cov_dim3}. Then the following assertions are equivalent: 
\begin{itemize}
    \item There exists a small step model $\mathcal S$ such that $\delta= {\cov}(\chi_{\mathcal S})$;
    \item The parameters satisfy $a^2 +b^2 +c^2 +2abc < 1$;
    \item The matrix $\delta$ is positive definite.
\end{itemize}
\end{thm} 

Note that a much weaker version of this statement is proved in \cite[Sec.~7.8]{BoPeRaTr-20}, where random walks are not assumed to have small steps, nor to be lattice random walks. Note also that we restrict the study to the case $a,b,c \geq 0$, which always holds in a Weyl chamber of a finite Coxeter group. Indeed, by  \cite[Lem.~4]{BoPeRaTr-20} or \cite[Prop.~4]{GoHuRa-25}, $a,b,c$ must have the  form  $\cos\Theta$, where $\Theta$ is  the angle between two walls of the Weyl chamber. From the classification of finite Coxeter groups (see again \cite[Sec.~6.2]{GoHuRa-25} for dimensions $3,4$ and \cite{Hu-90} for arbitrary dimensions), these angles have the form
$\pi/n$ for some $n \geq 2$ and thus have a non-negative cosine. 
%\textcolor{magenta}{Dire clairement à un moment que le Theorem~\ref{cov_presc} est aussi vrai quand les paramètres $a,b,c$ ne sont pas forcément positifs}

\subsection{A construction} 
\label{subsec:construction}
Before proving Theorem~\ref{cov_presc}, we present a key argument that allows us to construct a random walk with a prescribed covariance matrix. Consider any $2$-dimensional walk associated with an inventory $Q$ such that 
\begin{equation}
\label{eq:condition_cov_Q}
    {\cov}(Q) = \left( \begin{array}{cc} 
1 & -d \\
-d & 1 \end{array} \right)
\end{equation}
for some $0 \leq d <1$, for instance 
\begin{equation}
\label{eq:example_Q}
   Q(x,y)=\frac{d^2}{1-d^2}(x \overline{y} + \overline{x} y) + x +\overline{x},
\end{equation}
with $\overline{x}=\frac{1}{x}$ and $\overline{y}=\frac{1}{y}$. We also assume that its unique critical point is $(1,1)$, as it is the case for \eqref{eq:example_Q}. Then we prove:

\begin{prop}
\label{construction} 
Let $\alpha, \beta>0$. We consider the three-dimensional random walk $\mathcal S$ associated to the inventory
\begin{equation}
\label{eq:def_chi-Q}
    \chi_{\mathcal S}(x,y,z) = Q(x,y) + \alpha Q(z,y) + \beta Q(x,z).
\end{equation}
Then its covariance matrix is given by 
\begin{equation}
\label{delp}
{\cov}(\chi_{\mathcal S}) =\left( \begin{array}{ccc} 
1 & \frac{- d}{\sqrt{(1+\al)(1+\be)}} &  \frac{- d \be }{\sqrt{(1+\be)(\al+\be)}} \\
\frac{- d}{\sqrt{(1+\al)(1+\be)}}  & 1 & \frac{- d \al }{\sqrt{(1+\al)(\al+\be)}}\\
\frac{-d \be }{\sqrt{(1+\be)(\al+\be)}} & \frac{-d \al }{\sqrt{(1+\al)(\al+\be)} }  & 1 \end{array} \right).
\end{equation} 
\end{prop}

\begin{proof}
We first easily compute that 
$$\partial_x \chi_{\mathcal S}(1,1,1) = (1+\beta) \partial_x Q(1,1)=0.$$
In the same way, 
$\partial_y \chi_{\mathcal S}(1,1,1) = \partial_z \chi_{\mathcal S}(1,1,1) =0$.
Hence the unique critical point of $\chi_{\mathcal S}$ is $(1,1,1).$
Now observe that 
\begin{equation*}
    \left\{\begin{array}{rcl}
    \partial_{xx} \chi_{\mathcal S}(1,1,1)&=& (1+\be) \partial_{xx} Q(1,1),\smallskip\\
    \partial_{yy} \chi_{\mathcal S}(1,1,1)&=&(1+\al)  \partial_{yy} Q(1,1),\smallskip\\
    \partial_{zz} \chi_{\mathcal S}(1,1,1) &=&(\al+\be)  \partial_{xy} Q(1,1).
    \end{array}\right.
\end{equation*}
Since $ \frac{\partial_{xy} Q(1,1)}{\sqrt{\partial_{xx} Q(1,1)\partial_{yy} Q(1,1) }} =-d$, we get 
$$\frac{\partial_{xy} \chi_{\mathcal S}(1,1,1)}{\sqrt{\partial_{xx} \chi_{\mathcal S}(1,1,1)\partial_{yy} \chi_{\mathcal S}(1,1,1) }}=  
\frac{-d}{\sqrt{(1+\al)(1+\be)}}.$$
In the same way
$$\frac{\partial_{xz} \chi_{\mathcal S}(1,1,1)}{\sqrt{\partial_{xx} \chi_{\mathcal S}(1,1,1)\partial_{zz} \chi_{\mathcal S}(1,1,1) }}=  
\frac{- d \be}{\sqrt{(1+\be)(\al+\be)} }$$
and
$$\frac{\partial_{yz} \chi_{\mathcal S}(1,1,1)}{\sqrt{\partial_{yy} \chi_{\mathcal S}(1,1,1)\partial_{zz} \chi_{\mathcal S}(1,1,1) }}=  
\frac{- d \al}{\sqrt{(1+\al)(\al+\be)}},$$
which proves Proposition~\ref{construction}.
\end{proof}

\subsection{Proof of Theorem~\ref{cov_presc}}

First, from the definition, if there exists a model $\mathcal S$ such that $\delta= {\cov}(\chi_{\mathcal S})$, then necessarily $\delta$ is positive definite and then the determinant of $\delta$ satisfies
\begin{equation} 
\label{det>0}
    a^2 +b^2 +c^2 +2abc < 1.
\end{equation}
So we just need to prove that if \eqref{det>0} is true,
then we can find a model $\mathcal S$ such that $\delta= {\cov}(\chi_{\mathcal S})$. Assume that the equation \eqref{det>0} holds, that $\delta$ is not the identity matrix, and that (without loss of generality)\ $a \not= 0$. We can define $d \in (0,1)$ such that 
\begin{equation} 
\label{defd}
\frac{a^2}{d^2} + \frac{b^2}{d^2}+\frac{b^2}{d^2}+\frac{2abc}{d^3}=1.
\end{equation}
We then let $A:= \frac{a}{d}$,   $B:= \frac{b}{d}$ and  $C:= \frac{c}{d}$. It holds that 
\begin{equation} \label{ABC}
A^2 +B^2 +C^2 +2ABC =1.
\end{equation} 
Let $Q(x,y)$ be any inventory polynomial whose unique critical point is $(1,1)$ and such that \eqref{eq:condition_cov_Q} holds; for example, a suitable $Q$ is given in \eqref{eq:example_Q}.
Set
\begin{equation*}
    \al:= \left\{ \begin{array}{cc} \frac{C(1-A^2)}{A(AC+B)}>0 & \hbox{ if }  c \not=0\\ 0&  \hbox{otherwise} \end{array} \right.\qquad \text{and} \qquad
    \be:=\left\{ \begin{array}{cc} \frac{B(1-A^2)}{A(AB+C)}>0  & \hbox{ if }  b \not=0\\  0&  \hbox{otherwise}. \end{array} \right.
\end{equation*}
 Define the model $\mathcal S$ by its inventory $\chi_{\mathcal S}$ as in \eqref{eq:def_chi-Q}. By Proposition~\ref{construction}, we have \eqref{delp}.
We only consider the case where $a,b,c$ are all non-zero; the other cases are actually much simpler. By direct computation it holds that 
\begin{equation*}
%\label{albe}
1+\al = \frac{AB+C}{A(AC+B)}\quad \text{and}\quad 1+\be = \frac{AC+B}{A(AB+C)},
\end{equation*} 
which immediately gives  
$$\frac{d}{\sqrt{(1+\al)(1+\be)}}= dA =a.$$
Again by direct calculation, and in a second step using \eqref{ABC}, we deduce that 
$$\al +\be= \frac{(B^2+C^2+2ABC)(1-A^2)}{A (AB+C)(AC+B)}= \frac{(1-A^2)^2}{A (AB+C)(AC+B)},$$
from which we obtain 
$$\frac{d \be}{\sqrt{(1+ \be)(\al +\be)}}=dB=b \quad \text{and} \quad \frac{d \al}{\sqrt{(1+ \al)(\al +\be)}}=dC=c.$$
Together with \eqref{delp}, this proves Theorem~\ref{cov_presc}.\qed

\subsection{Examples} 
\label{exam}
The proof of Theorem~\ref{cov_presc} gives a constructive way to produce examples of models with a given covariance matrix, and therefore a given reflection group. We give here two examples. 

\medskip

Let $H$ (resp.\ $H'$) be the reflection group induced by $Q(x,y)$ in \eqref{eq:example_Q} (resp.\ $\chi_{\mathcal S}(x,y,z)$ in \eqref{eq:def_chi-Q}). We recall that $H$  (resp.\ $H'$) is the group spanned by the reflections through the walls of the orthant $C:= {\cov}(Q)^{-\frac{1}{2}} \bigl(\R_+^2\bigr) $  (resp.\ $C':= {\cov}(\chi_{\mathcal S})^{-\frac{1}{2}} \bigl(\R_+^3\bigr) $). See the introduction or \cite{GoHuRa-25} for more details on the construction of these reflection groups. 

\begin{ex} 
\label{exB3}
\textnormal{Let $a= \cos(\pi/4)$, $b=\cos(\pi/3)$ and $c=0$ in \eqref{eq:cov_dim3}. We refer to the classification  given for instance in \cite[Sec.~6.2]{GoHuRa-25}:  $H'$ is isomorphic to the group $B_3$ (with the standard notations of the classification of finite Coxeter groups). Note that $B_3$ is itself isomorphic  to $\Z_2 \times \mathcal{S}_4$, where $\mathcal{S}_4$ is the standard  permutation group on four elements. 
 %\textcolor{magenta}{Ajouter un petit argument, voire même une définition pour $\mathcal S_4$} 
Then using \eqref{defd} we find $d=\cos(\pi/6)$, and we get such a model by starting from \eqref{eq:example_Q}
$$Q(x,y)=3(x \overline{y} + \overline{x}y) +x + \overline{x},$$
whose associated reflection group $H$ is isomorphic to the dihedral group $D_{12}$. Note that any inventory $Q(x,y)$ such that $H$ is isomorphic to $D_{12}$ gives rise to a group $H'$ which is isomorphic to $B_3$. This remark will be useful in the next section.}
\end{ex}

\begin{ex}
\textnormal{From Theorem~\ref{cov_presc}, we know that there exists an inventory $\chi_{\mathcal S}(x,y,z)$ with $a= \cos(\pi/5)$, $b=\cos(\pi/3)$ and $c=0$, since $a^2+b^2 <1$. The associated reflection group is a finite Coxeter group, often called $H_3$ in the classification of Coxeter groups, which is the isometry group of the icosahedron in $\R^3$. To our knowledge, this is the first example whose reflection group is $H_3$: indeed, $H_3$ never appears as a reflection group in the classification of unweighted models given in Table~1 of \cite{BaKaYa-16} (unweighted means that all non-zero weights $a(s)$ are equal).}
\end{ex}

\subsection{One-parameter families} 

In Theorem~\ref{cov_presc} we were able to construct a random walk with a prescribed covariance matrix. In this section, we go further and construct one-parameter families of zero-drift models with a prescribed covariance matrix in dimension two and three. Our motivation is to show that our construction is robust and also to construct as many examples as possible in the hope that some of them may have very special properties.

We start with the dimension $2$. Let $\Theta \in (0, \frac{\pi}{2})$.  We want to find a one-parameter family of models $Q_t$ in dimension $2$ such that 
\begin{equation}
    \label{eq:covQt}
    {\cov}(Q_t)=  \left( \begin{array}{cc} 
1 & -\cos\Theta  \\
-\cos\Theta & 1  \end{array} \right).
\end{equation}
We let $t \in (0,\sin^2\Theta)$ and we apply the construction of Section~\ref{subsec:construction} with $a=t$, $b=0$, $c=\cos\Theta$ and \eqref{eq:example_Q} with
$d= \sqrt{\cos^2\Theta + t^2}$. We then set 
$$\widetilde{Q}_t(x,y)= \frac{d^2}{1-d^2}(x \overline{y} + \overline{x}y) +  x +\overline{x} = \frac{\cos^2\Theta+t^2}{\sin^2\Theta -t^2} (x \overline{y} + \overline{x}y) +  x +\overline{x}.$$
We then compute 
\begin{equation*}
    A= \frac{a}{d} = \frac{t}{\sqrt{\cos^2\Theta+ t^2}},\quad B=\frac{b}{d}=0\quad \text{and}\quad C=\frac{c}{d}= \frac{\cos\Theta}{\sqrt{\cos^2\Theta+ t^2}}.
\end{equation*}
This gives $\beta=0$ and $\al=\frac{1-A^2}{A^2} = \frac{\cos^2\Theta}{t^2}$. 
Now as in \eqref{eq:def_chi-Q} we set $P_t(x,y,z)=\widetilde{Q}_t(x,y) + \al\widetilde{Q}_{t}(z,y)$, i.e., 
$$P_t(x,y,z)=\frac{\cos^2\Theta+t^2}{\sin^2\Theta -t^2} (x \overline{y} + \overline{x}y) +  x +\overline{x} + \frac{(\cos^2\Theta+t^2)\cos^2\Theta}{t^2(\sin^2\Theta -t^2)} (z \overline{y} + \overline{z}y) + \frac{\cos^2\Theta}{t^2} ( z +\overline{z}), $$
and we obtain a one-parameter family of models whose covariance matrix is 
\begin{equation}
{\cov}(P_t) =\left( \begin{array}{ccc} 
1 & - t &  0 \\
- t  & 1 & - \cos\Theta\\
0 & - \cos\Theta & 1 \end{array} \right).
\end{equation} 
We set $Q'_t(x,y)=P_t(1,x,y)$. Then its covariance matrix ${\cov}(Q'_t)$ must be the $2\times2$ block in the lower right of ${\cov}(P_t)$ and is thus given by \eqref{eq:covQt}.

Actually, we can remove the constant term in $Q'_t$ and multiply by $\frac{t^2(\sin^2\Theta -t^2)}{\cos^2\Theta +t^2}$  without modifying the covariance matrix. This allows to obtain simpler constants in the model. In this way we obtain  
\begin{equation*}
Q_t(x,y)= \frac{t^2(\sin^2\Theta -t^2)}{\cos^2\Theta +t^2}(Q'_t(x,y)-2)  =  \cos^2\Theta (x \overline{y} + \overline{x} y )  + t^2 (x + \overline{x}) + \frac{\cos^2\Theta (\sin^2\Theta -t^2)}{\cos^2\Theta +t^2}(y+ \overline{y}),
\end{equation*}
which gives the desired one-parameter family of models in  dimension $2$ having all the same covariance matrix
as in \eqref{eq:covQt}.

\medskip

Next, assume that the three-dimensional covariance matrix $\delta$ is as in \eqref{eq:cov_dim3}, with given parameters $a,b,c$ such that $\delta$ is positive definite. We can do the construction explained in Section~\ref{exam}, but replacing $Q$ with the one-parameter family of models $Q_t$ we just constructed, which has the same covariance matrix as $Q$. We thus obtain a one-parameter family of models whose covariance matrix is $\delta$.

\begin{ex}
\textnormal{We illustrate these constructions both in dimension $2$ and $3$ by fixing the explicit value $\Theta= \pi/6$. 
The construction above gives
$$Q_t(x,y)= \frac{3}{4} (\overline{x} y + x \overline{y}) + t^2(x + \overline{x}) + \frac{3-12t^2}{12+16t^2} (y+\overline{y}).$$
Using the construction of Example~\ref{exB3}, we set 
$$\begin{aligned} 
R_t(x,y,z)  & = Q_t(x,y) +2Q_t(z,y) \\
& =\frac{3}{4} (\overline{x} y + x \overline{y}) + t^2(x + \overline{x}) + \frac{9-36t^2}{12+16t^2} (y+\overline{y})+\frac{3}{2} (\overline{z} y + z \overline{y}) +2 t^2(z + \overline{z}). \end{aligned}$$ 
We thus obtain a one-parameter family with covariance matrix 
\begin{equation*}
{\cov}(R_t) =\left( \begin{array}{ccc} 
1 & - \cos(\pi/4) &  \cos(\pi/3) \\
- \cos(\pi_4) & 1 &  0\\
-\cos(\pi/3) & 0 & 1 \end{array} \right),
\end{equation*}
whose associated reflection group is isomorphic to $B_3$. We recall that $B_3$ is a classical notation in the classification of finite Coxeter groups and is isomorphic to $\Z_2 \times \mathcal{S}_4$.}
\end{ex}

As an ending remark, we notice that the list of  examples above is far from being exhaustive. Our goal is to show that each  model in dimension $2$ gives rise to families of models in dimension $2$ and $3$. For instance, we could have started the construction above with 
$$Q(x,y)= \frac{r+s(1-2d)+4d}{2(1-d)}(\overline{x} y + x \overline{y}) + (1-r) (x +y) +(\overline{x} + \overline{y}) + (r+s) xy + (1-s)
\overline{x}\overline{y}$$
instead of 
$$ \frac{d^2}{1-d^2}(x \overline{y} + \overline{x}y) +  x +\overline{x} = \frac{\cos^2\Theta+t^2}{\sin^2\Theta -t^2} (x \overline{y} + \overline{x}y) +  x +\overline{x},$$ since they both have the covariance matrix $$  \left( \begin{array}{cc} 
1 & -d  \\
-d & 1  \end{array} \right).$$
We would have obtained in this way a more general three-parameter family of models whose covariance matrix is exactly as in \eqref{eq:covQt}.
We do not make this computation here.

\section*{Acknowledgments}
EH is supported by the project Einstein-PPF (\href{https://anr.fr/Project-ANR-23-CE40-0010}{ANR-23-CE40-0010}), funded by the French National Research Agency. KR is supported by the project RAWABRANCH (\href{https://anr.fr/Project-ANR-23-CE40-0008}{ANR-23-CE40-0008}), funded by the French National Research Agency.   We are very grateful to Pierre Tarrago: the arguments used to prove Theorem~\ref{comparingfP} are his ideas. Finally, we would like to thank Denis Denisov, Nikita Elizarov and Vitali Wachtel for interesting discussions.

\bibliographystyle{siam}
\bibliography{biblio}

\end{document}